\documentclass{amsproc}
\usepackage{amssymb}
\usepackage{amsmath}
\usepackage{amsthm}
\usepackage[a4paper]{geometry}
\newcommand{\be}[1]{\begin{eqnarray#1}}
\newcommand{\ee}[1]{\end{eqnarray#1}}
\newtheorem{thm}{Theorem}[section]
\newtheorem{propn}[thm]{Proposition}
\newtheorem{lemma}[thm]{Lemma}

\theoremstyle{remark}
\newtheorem{remark}[thm]{Remark}
\newtheorem{definition}[thm]{Definition}

\newcommand{\tp}{\otimes}

\newcommand{\braid}[2]{{#1}$\lower4pt\hbox{$\tp\atop\raise4pt
            \hbox{$\scriptscriptstyle\Ru $}$}${#2}}
\newcommand{\twist}[2]{{#1}${\,\scriptscriptstyle \Ru}\atop\raise9pt\hbox{$\scriptstyle\tp$} ${#2}}
\newcommand{\twistF}[2]{{#1}${\,\scriptscriptstyle \F}\atop\raise9pt\hbox{$\scriptstyle\tp$} ${#2}}

\newcommand{\Bc}{\mathcal{B}}
\newcommand{\Ac}{\mathcal{A}}
\newcommand{\Kc}{\mathcal{K}}

\newcommand{\op}{\oplus}

\newcommand{\A}{\mathcal{A}}

\newcommand{\Ru}{\mathcal{R}}

\newcommand{\F}{\mathcal{F}}

\newcommand{\s}{\mathfrak{s}}

\newcommand{\C}{\mathbb{C}}
\newcommand{\Z}{\mathbb{Z}}

\newcommand{\N}{\mathbb{N}}

\newcommand{\g}{\mathfrak{g}}

\renewcommand{\l}{\mathfrak{l}}
\newcommand{\p}{\mathfrak{p}}

\newcommand{\la}{\lambda}

\newcommand{\n}{\nonumber          }
\newcommand{\al}{\alpha}
\newcommand{\bt}{\beta}
\newcommand{\dt}{\delta}
\newcommand{\id}{\mbox{id}}

\newcommand{\End}{\mathrm{End}}

\newcommand{\Span}{\mathrm{Span}}

\newcommand{\si}{\sigma}

\begin{document}
\title{Solutions to graded reflection equation of $GL$-type}
\author{D. Algethami${}^{\dag,\ddag}$, A. Mudrov${}^{\dag,\sharp}$, V. Stukopin${}^\sharp$}
\address{${}^{\dag}$ University of Leicester,
 University Road,
LE1 7RH Leicester, UK,}
\address{
 ${\ddag}$ University of Bisha,
 255, Al Nakhil,
67714 Bisha, Saudi Arabia,
}{
\address{
${}^{\sharp}$ MIPT,
9 Institutskiy per., Dolgoprudny, Moscow Region,
141701, Russia}
\email{daaa3@leicester.ac.uk, am405@le.ac.uk, stukopin.va@mipt.ru}
\maketitle
\begin{abstract}
We list solutions of the graded reflection
equation associated with the fundamental vector representation
of a quantum supergroup of $GL$-type.
\end{abstract}
{\small \underline{Key words}: Quantum general linear supergroup, Graded reflection equation, Baxterization}\\

{\small \underline{Mathematics Subject Classification}: 17B37, 81R50}
\section{Introduction}
There are two fundamental equations in the theory of  quantum integrable models.
The first one is the Yang-Baxter equation (YBE) which
is a key ingredient of quantum inverse scattering method; it  gave rise to quantum groups \cite{FRT,Dr}. The other is
Reflection equation (RE) also known as the boundary YBE \cite{KSkl,KS,KSS}. It appears in  models with open boundary conditions \cite{Cher,Skl} and is  related with coideal subalgebras and quantum homogeneous spaces \cite{Dij,NDS,N,M1}. The theory of quantum symmetric pairs is developed in great detail in \cite{Let,K} and
the corresponding universal   solutions to RE have been constructed in \cite{BK,RV}.

Reflection equation associated with infinite dimensional quantum groups lead to solutions (K-matrices) depending on spectral parameter,
in contrast with constant solutions, which are related to finite dimensional quantum groups. Classification of K-matrices
is a fundamental problem, which is completely solved to date only for $\g\l(N)$:  for the constant case in \cite{M1} and for
those (invertible) with spectral parameter  in \cite{RV1}. Although the constant version of RE is of less demand in applications to physics, it plays a role in such a classification
thanks to a baxterization procedure \cite{KM,IO}. It has been proved in \cite{RV1} that every invertible K-matrix of the affine quantum group $U_q\bigl(\widehat{\g\l(N)}\bigr)$ can be
obtained from a baxterized constant K-matrix of  $U_q\bigl(\g\l(N)\bigr)$  via a family of transformations called equivalence. Thus \cite{RV1} is not just listing K-matrices
with spectral parameter but clarifies their structure.
In this context,  classification of constant K-matrices becomes an important first step.

A supersymmetric  version of  RE was addressed in the literature with  a focus on affine quantum supergroups and super Yangians \cite{AACLFR,AACLFR1,DK,L,R}
and for  special gradings of the
underlying vector space. It
is therefore interesting to obtain a classification of constant solutions to RE for an arbitrary grading, as a first step
of the classification programme of \cite{RV1}.
We do this for the general linear  quantum supergroups.

We give a complete classification of the solutions to RE of the aforementioned type. Our method is motivated by an earlier  result of \cite{M1} for the non-graded case.
It turns out that
solutions of the graded RE are exactly even matrices that solve the non-graded RE. When it comes to non-degenerate K-matrices, solutions exist
only for a special "symmetric" grading that is preserved by  the longest element of the symmetric group acting on a weight basis of the defining representation.

In the next section, we give basics of the graded RE and prove its equivalence to a non-graded RE with an appropriate
R-matrix. For even solutions, that is exactly the R-matrix satisfying the  equivalent non-graded YBE equations.
We use that correspondence and reduce the study to a non-graded RE. That allowed us to adapt the reasoning
 of \cite{M1} to the current exposition. We closely follow \cite{M1} making appropriate modifications as per grading,
 along with some additions and clarifications. The main result is stated in Section \ref{Sec_Classification}.
The proof is arranged in a sequence of 9 lemmas in Section \ref{Sec_Proof}.
Baxterization of constant solutions to spectral parameter dependent RE solutions is considered in Section \ref{Sec_Baxt}.
\section{$\Z_2$-graded Yang-Baxter and Reflection equations}
Let us recall basic definitions of graded algebras and modules.

An associative algebra $\Ac$ is called $\Z_2$-graded if it is presentable as sum of vector spaces $\Ac=\Ac_0\oplus \Ac_1$ satisfying $\Ac_i\Ac_j\subset \Ac_{i+j\mod 2}$.
In particular,  $\Ac_0$ is a subalgebra and $\Ac_1$ is a left/right $\Ac_0$-module. Elements of $\Ac_0$ are called even and elements of $\Ac_1$ are called odd.
An $\Ac$-module $V$ is called graded if it decomposes to a direct sum $V=V_0\op V_1$ such that $\Ac_i V_j\subset V_{i+j\mod 2}$.

An example is a complex vector space $V=\C^N$ with a basis $\{v_i\}_{i=1}^N$ equipped with an arbitrary parity function $\{1,\ldots,N\}\to \{0,1\}$,
$i\mapsto [i]=\deg(v_i)$.
Then $V_i=\Span\{e_k\}_{[k]=i}$, $i=0,1$. It is a graded module over the algebra $\Ac=\End(V)$ whose graded components
are set to be
$\Ac_i=\Span\{e_{lk}\}_{[l]+[k]=i\mod 2}$, $i=0,1$. Here $e_{lk}$ are the matrix units acting on the basis vectors by  $e_{lk}v_j=\dt_{kj}v_l$.

Given two  graded algebras $\Ac$ and $\Bc$  their tensor product $\A\tp \Bc$ is a graded algebra too.
The multiplication on homogeneous elements is defined by the rule
$$
(a_1\tp b_1)(a_2\tp b_2)=(-1)^{\deg(a_2)\deg(b_1)}a_1 a_2\tp b_1b_2.
$$
One has $(\Ac\tp \Bc)_0=(\Ac_0\tp \Bc_0)\op(\Ac_1\tp \Bc_1)$ and $(\Ac\tp \Bc)_1=(\Ac_1\tp \Bc_0)\op(\Ac_0\tp \Bc_1)$.

In the special case when $\Ac=\Bc=\End(V)$, there is a $\Z_2$-graded isomorphism
\begin{equation}
\varphi\colon \End(V)\tp \End(V)\to \End(V\tp V), \quad e_{ij}\tp e_{lk}\mapsto (-1)^{[j]([l]+[k])} e_{il,jk},
\label{iso-phi}
\end{equation}
where $e_{il,jk}\in \End(V\tp V)$ are matrix units acting by  $e_{il,jk}(v_{m}\tp v_n)=\dt_{lm}\dt_{kn}v_{i}\tp v_j$.
Given $F=\sum_{ijlk}F_{ij,lk}e_{ij}\tp e_{lk}\in \End(V)\tp \End(V)$ we will write $\tilde F=\sum_{ijlk}\tilde F^{il}_{jk}e_{il,jk} \in \End(V\tp V)$ for its $\varphi$-image.
The entries are related by $\tilde F^{il}_{jk}=(-1)^{[j]([l]+[k])}F_{ij,lk}$.

Suppose $V$ is a graded vector space and $\Ac=\End(V)$ is the corresponding graded matrix algebra.
An invertible element $R\in \Ac\tp \Ac$ is called an $R$-matrix if it satisfies  Yang-Baxter equation
$$
R_{12}R_{13}R_{23}=R_{23}R_{13}R_{12},
$$
where the subscripts indicate the tensor factor in the graded tensor cube of $\End(V)$.
It is usually assumed that $R$ is even. In terms of the matrix $\tilde R\in \End(V\tp V)$, the YBE reads
$$
 \sum_{a,b,c}\tilde R_{ca}^{lm}\tilde R_{ib}^{cn}\tilde R_{jk}^{ab} (-1)^{[a][b]+[n][a]}=
 \sum_{a,b,c}\tilde R_{bc}^{mn}\tilde R_{ak}^{lc}\tilde R_{ij}^{ab}(-1)^{[b][c]+[k][b]}.
$$

An even element $P=\sum_{i=1}^{n}(-1)^{[j]}e_{ij}\tp e_{ji}\in \End(V)\tp \End(V)$ is called graded permutation. It features
$$
P(v\tp w)=(-1)^{\deg(v)\deg(w)}w\tp v
$$
for all homogeneous $v,w\in V$.
Then the operator $S=PR\in \End(V)\tp \End(V)$ satisfies the braid relation
$$
S_{12}S_{23}S_{12}=S_{23}S_{12}S_{23}.
$$
A matrix $A\in \End(V)$ is said to satisfy the (graded) reflection equation if
$$
SA_2S A_2= A_2 S A_2 S
$$
in $\End(V)\tp \End(V)$.
We will generally not assume that $A$ is even.

It is known that graded YBE is equivalent to the non-graded YBE
with the R-matrix with entries $\breve R^{ij}_{kl}=(-1)^{[i][j]+[j][k]+[k][l]}R_{ik,jl   }$. Respectively, the graded braid  equation with matrix $S$ goes over to
the non-graded braid  equation with matrix elements $\breve S^{ij}_{kl}=\breve R^{ji}_{kl}$.
\begin{propn}
The graded RE  on   $A\in \End(V)$  is equivalent to the non-graded RE
\begin{equation}\label{re}
\hat S A_2 \hat S A_2 =A_2 \hat S A_2 \hat S,
\end{equation}
with $\hat S^{ij}_{kl}=\breve S^{ij}_{kl}(-1)^{[i][j]+[k][l]}$, and $\hat S^{ij}_{kl}=\breve S^{ij}_{kl}$ under the assumption that $A$ is  even.
\end{propn}
\begin{proof}
  Apply the isomorphism $\varphi$ defined in (\ref{iso-phi}) to the graded RE and get
$$
\tilde S\tilde A_2\tilde S \tilde A_2=\tilde A_2\tilde S \tilde A_2\tilde S.
$$
Observe that $\tilde S=\breve S$. Plugging in the entries of the tensor $\tilde A_2$ expressed through the matrix $A$,   rewrite the equation as
$$
\sum_{a,b,c,d}\tilde S^{ij}_{ab}A_{bc}(-1)^{[a]([b]+[c])}\tilde S^{ac}_{md}(-1)^{[m]([d]+[n])} A_{dn}=\hspace{150pt}
$$
\vspace{-20pt}
$$
\hspace{150pt}
=
\sum_{a,b,c,d}A_{ja}(-1)^{[i]([j]+[a])}\tilde S^{ia}_{bc} A_{cd}(-1)^{[b]([c]+[d])}\tilde S^{bd}_{mn}.
$$
This turns to the non-graded RE with with braid matrix $\tilde S$ if $A$ is even. For arbitrary $A$,
multiplication of the equation by $(-1)^{[i][j]+[m][n]}$ proves the statement in general.
\end{proof}
\section{Classification of solutions}
\label{Sec_Classification}
We fix a graded $R$ matrix of the $GL$-type in  the fundamental vector representation
 $V=\C^N$ with  an arbitrary grading on the the basis elements:
$$
R=\sum_{i,j} e_{ii}\tp e_{jj}q^{(-1)^{[i]}\dt_{ij}}+\omega\sum_{i>j}e_{ij}\tp e_{ji}(-1)^{[j]}\in \End(V)\tp \End(V),
$$
where $\omega$ stands for $=q-q^{-1}$.  This matrix differs from the one in \cite{Zh} in the sign of inequality in the second sum.

In what follows, we will work with the equivalent non-graded R-matrix and use the symbol $\tp$ to denote  the non-graded tensor product.
In order to make double indexing more readable we will use  upper and lower indices. The matrix units $e_{ij}$ will be  replaced with $e^i_j$.

The equivalent non-graded R-matrix is
$$
\breve R=\sum_{i,j} (-1)^{[i][j]}q^{(-1)^{[i]}\dt_{ij}}e^{i}_j\tp e^i_{j}+\omega\sum_{i<j}e^i_j\tp e^j_i.
$$
 It is convenient to represent the corresponding  braid  matrix
in the form
\be{}
\label{S}
\hat S =\breve S &=&\sum_{i,j} (-1)^{[i][j]}q^{(-1)^i\dt_{ij}}e^j_i\tp e^i_j +\omega\sum_{i<j}e^j_j\tp e^i_i\>
\n\\
&=&
\sum_{i,k} s_{ik}e^{ik}_{ik}+\sum_{i\not =j} (-1)^{[i][j]}e^{ij}_{ji},\quad
\mbox{where}\;
s_{ik}=
\left\{
\begin{array}{ll}
\omega,& i <k,\\
(-1)^{i}q^{(-1)^i},& i =k,\\
0,& i >k.
\end{array}
\right.
\n
\ee{}
The transposed matrix $\hat S_{21}$ is equivalent to the graded R-matrix from \cite{Zh}.

We denote  by $[a,b]\subset \Z$ the intervals
$\{k\in \Z| a\leq k \leq b\}$. Respectively we use parentheses for the
intervals defined by strict inequalities.
Classification theorem is formulated with the use of  the following data.
\noindent
\begin{definition}
\label{ap}
{\em An admissible  pair} $(Y,\si)$ consists of an
ordered subset
$Y\subset I=[1,N]$  and a parity preserving strictly decreasing  map
$\si\colon Y \to I$ without stable points.
\end{definition}
\noindent
%The map $\si$ is determined by its image subject to $\si(Y)\cap Y=\varnothing$.
We distinguish two
non-intersecting subsets $Y_+=\{i\in Y| i>\si(i)\}$
and $Y_-=\{i\in Y| i<\si(i)\}$; obviously $Y=Y_-\cup Y_+$. Denote
$b_-=\max\{Y_-\cup \si (Y_+)\}$ and  $b_+=\min\{Y_+\cup \si (Y_-)\}$.
Because the $\si$ is decreasing, we have $b_-<b_+$.
We adopt the convention $b_-=0$ and $b_+=N+1$ if
$Y=\varnothing$.
\begin{thm}
\label{clth}
General  solution to the graded RE of the $GL$-type is a matrix $A$ of  the form
\be{}
\label{gf}
A&=&\sum_{i\in I} x_i e^i_i + \sum_{j\in Y}y_j  e^{\si(j)}_j,
\hspace{28pt}
\ee{}
where
\begin{itemize}
\item $(Y,\si)$ is an  admissible pair with
\be{}
Y&=&[1, b_-]\cup [b_+, b_+ + b_- -1] ,\hspace{20pt}
\\
\si(i)&=&b_+ + b_- -i
\ee{}
for $b_-,b_+\in \N$ subject to the conditions
$b_-<b_+$, $b_- + b_+\leq N+1$ and
\be{}
x_i & = &
\left\{
\begin{array}{ll}
\la+\mu,& i\in [1,b_-],\\
\la,& i\in (b_-,b_+),\\
0,& i\in [b_+,N],
\end{array}
\right.\hspace{35pt}\\
y_i y_{\si(i)} & = &-\la\mu \not =0
\ee{}
for  $\la,\mu \in \C$;
\item  $(Y,\si)$ is an admissible pair such that
 $Y \cap \si(Y)=\varnothing$ and
\be{}
x_i & = &
\left\{
\begin{array}{cc}
\la,& i\in [1,b],\\
0,& i\in (b,N],
\end{array}
\right.\hspace{42pt}\\
y_i &  \not= &0
\ee{}
for $b\in [b_-,b_+)$ and $\la\in \C$.
\end{itemize}
\end{thm}

As follows from the theorem, there are two classes of
numerical RE matrices, those corresponding to
$\si(Y)=Y$ and $\si(Y)\cap Y=\varnothing$. We call them
solutions of Type 1 and 2, respectively.

Theorem \ref{clth} asserts that solutions of graded RE are exactly even matrices  from \cite{M1} solving the non-graded RE of $\g\l(N)$.
Different gradings impose different restricting conditions on $A$. So, invertible $A$ (they belong to Type 1) are even only for
a "symmetric" grading of  the underlying vector space:
\begin{center}
$
\left(\underbrace{0, \ldots, 0}_\ell ;\underbrace{1,\ldots,1}_{N-2\ell},\underbrace{0, \ldots, 0}_\ell\right)
$
\end{center}
or its opposite.
\section{Proof of the classification theorem}
\label{Sec_Proof}
The proof of Theorem (\ref{clth}) formulated in the previous section
is combinatorial. It is a result of the direct analysis of
equation (\ref{re}) organized into a sequence of lemmas.
Explicitly  (\ref{re}) reads
\be{}
\hat SA_2\hat SA_2&=&
\sum_{i,\bt,\al,\nu} s_{i\bt}s_{i\nu} A^{\nu}_\bt A^\al_{\nu} \; e^i_i\tp e^\bt_\al
+
\sum_{i,j,\bt,\al \atop j\not = i} (-1)^{[i][j]}s_{i\al}A^{j}_\al A^\bt_i   \; e^i_j\tp e^\al_\bt
\label{lhs}\\
&+&
\sum_{i,\bt,\al,\nu \atop i\not = \bt} (-1)^{[i][\bt]} s_{i\nu} A^{\nu}_\bt A^\al_{\nu} \; e^\bt_i\tp e^i_\al
+
\sum_{i,j,\bt,\al \atop {j\not = i \not = \al }}  (-1)^{[i]([\al]+[j])} A^{j}_\al A^\bt_i   \; e^\al_j\tp e^i_\bt
,\n\\
A_2\hat SA_2\hat S&=&
\sum_{i,\bt,\al,\nu} s_{i\al}s_{i\nu} A^{\nu}_\bt A^\al_{\nu} \; e^i_i\tp e^\bt_\al
+
\sum_{i,j,\bt,\al \atop j\not = i} (-1)^{[i][j]} s_{j\bt}A^{j}_\al A^\bt_i   \; e^i_j\tp e^\al_\bt
\label{rhs}\\
&+&
\sum_{i,\bt,\al,\nu \atop i\not = \al}(-1)^{[i][\al]} s_{i\nu} A^{\nu}_\bt A^\al_{\nu} \; e^i_\al\tp e^\bt_i
+
\sum_{i,j,\bt,\al \atop {i\not = j \not = \bt }} (-1)^{[i]([\bt]+[j])}A^{j}_\al A^\bt_i   \; e^i_\bt\tp e^\al_j
\n.
\ee{}
Comparison of similar terms   gives rise to the system
of quadratic equations on the matrix elements $A^i_j$.
\begin{lemma}
A general  solution to (\ref{re}) is an even matrix $A$ satisfying the following system of equations:
\be{}
\left\{
\begin{array}{rrl}
A^m_i A^n_i &=&0,
 \\
A^i_m A^i_n &=&0, %& m\not = n  \not = i \not = m
\end{array}
\right.
\hspace{44pt}
\quad  m\not = n  \not = i \not = m ,
\hspace{20pt}
\label{eq1}
\\
A^n_i A^j_m \hspace{8pt}=\hspace{8pt}0,
\hspace{36pt}
\quad
\left\{
\begin{array}{l}
j\not = m \not = n\not =i,\\
(m-i)(n-j)<0,
\end{array}
\right.
\label{eq2}
\\
\left\{
\begin{array}{rrll}
((-1)^{[i]})q^{(-1)^{[i]})} - s_{mi}) A^i_i A^m_i &=& \sum_{\nu}s_{i\nu} A^{\nu}_i A^m_{\nu},& i \not = m
, \\
((-1)^{[i]})q^{(-1)^{[i]})} - s_{mi}) A^i_i A^i_m &=& \sum_{\nu}s_{i\nu} A^{\nu}_m A^i_{\nu},& i \not = m
, \\
0&=&\sum_{\nu}s_{i\nu} A^{\nu}_m A^n_{\nu} ,&
(m-i)(n-i)<0,
\end{array}
\right.
 \label{eq3}\\
(-1)^{[i][n]}A^n_i A^i_m - \sum_{\nu}s_{i\nu} A^{\nu}_m A^n_{\nu}=
(s_{ni} - s_{im}) A^i_i A^n_m
, \quad  m\not = i  \not = n \not = m
, \hspace{8pt}
\label{eq4}
\\
\omega A^m_m A^i_i
=
\sum_{\nu}s_{i\nu} A^{\nu}_m A^m_{\nu}  -  \sum_{\nu}s_{m\nu}  A^{\nu}_i
A^i_{\nu}
, \quad  i< m.
%\hspace{pt}
\label{eq5}
\ee{}
\end{lemma}
Using a chance we point out a typo in $s_{mi}$ (order of indices) in (\ref{eq3}) admitted in \cite{M1}. This presentational flaw did not affect  Lemma 3.3 because it was based on the correct expressions.
\begin{proof}
The proof is done by a direct analysis of which we indicate the key steps.
Comparison of  coefficients before $e_j^i\tp e_j^i$ implies that $A$ is even.
Equations (\ref{eq1}) results from comparing terms with $e^i_m\tp e^i_n$ and $e_i^m\tp e_i^n$.
Equation (\ref{eq2}) with  $i\not =j$ is obtained from $e^i_j\tp e^m_n$. For $i=j$, it coincides with  (\ref{eq4}) modulo
third    line of (\ref{eq3}). The former comes from $e^i_n\tp e_i^m$ and the latter from $e^i_i\tp e_m^n$.
The sign factor in (\ref{eq4}) is reduced from  $(-1)^{[m][i]+[m][n]+[i][n]}$ because $A$ is even.
The 1st line of (\ref{eq3}) comes from $e^i_m\tp e^i_n$ while the 2nd from $e_j^m\tp e_j^n$.
Equation (\ref{eq5}) results from $e_i^m\tp e_m^i$. All other equations arising from (\ref{re}) are fulfilled in virtue of this system.
\end{proof}
The next lemma accounts for equations  (\ref{eq1})
and  (\ref{eq2}).
\begin{lemma}
\label{l-eq1}
If $A$ a solution to equation (\ref{re}), then it
can be presented in the form (\ref{gf}),
%$$A=\sum_{i\in I} x_j e^j_j + \sum_{i\in Y}y_i e^i_{\si(i)},$$
where $(Y,\si)$ is an admissible pair, $x_i=A^i_i$  for $i\in I$,
and  $y_i=A^{\si(i)}_i\not = 0$ for $i\in Y$.
\end{lemma}
\begin{proof}
By virtue of equations (\ref{eq1}), the matrix $A$ has
at most one non-zero off-diagonal entry in
every row and every column.
Indices of such rows  form a subset $Y$ in $I$,
and the  off-diagonal entries may be written as
$y_i=A_i^{\si(i)}\not = 0$,  $i\in Y$, for some
bijection $\si\colon Y\to I$ with no stable points.
This is equivalent to equation (\ref{eq1}). The
map $\si$ is decreasing, that is encoded in equation
(\ref{eq2}). It preserves parity because $A$ is even.
\end{proof}
\begin{lemma}
\label{l-eq31}
The system (\ref{eq3})
is equivalent to
\be{}
\label{eq3'}
\sum_{\nu\geq\max(i,m)}A^\nu_i A^m_\nu=0,\quad i\not = m,
\ee{}
by virtue of (\ref{eq2}). They imply $x_i=0$ for  $i = b_+$.
\end{lemma}
\begin{proof}
The proof is similar to non-graded  case because
subsystem (\ref{eq3}) turns out to be the same: it is equivalent to
$$
\left\{\begin{array}{rrll}
 0&=& \sum_{i<\nu}A^{\nu}_i A^m_{\nu},& m>i,\\
 0&=& \sum_{i\leq \nu} A^{\nu}_i A^m_{\nu},& i >m,\\
 0&=& \sum_{m<\nu}A^{\nu}_i A^m_{\nu},& i>m,\\
 0&=& \sum_{m\leq\nu} A^{\nu}_i A^m_{\nu},& m > i
, \\
0&=&\sum_{k\leq\nu}A^{\nu}_i A^m_{\nu} ,&
(i-k)(m-k)<0,
\end{array}
\right.
$$
Setting $i=j$ in (\ref{eq2}) and substituting it into
the above system   we reduce it
to (\ref{eq3'}).
Setting  $m =\si(i) < i \in Y_+$ in (\ref{eq3'})
leads to $ x_i y_i =0$ and therefore $x_i=0$.
Similarly, the assumption  $m=\si(i)>i \in Y_-$ reduces
(\ref{eq3'}) to $ y_i x_{\si(i)}=0$ and hence $x_{\si(i)}=0$.
So $x_i=0$ for all $i\in Y_+\cup \si(Y_-)$
and $i=b_+$ in particular.
\end{proof}
\begin{lemma}
\label{l-eq4}
With Lemma \ref{l-eq31} taken into account,
equation (\ref{eq4}) is equivalent to the following two assertions.
\begin{enumerate}
\item For any $m\in Y$, either $\si^2(m)=m$ or  $\si(m)\not\in Y$.
\item $x_i = x_{b_-} $ and $x_j = 0$ whenever $i\leq b_-$ and $j\geq b_+ $.
\end{enumerate}
\end{lemma}
\begin{proof}
First note that if $m\not \in Y$, equation (\ref{eq4})
holds identically.
So we can assume $m \in Y$ and rewrite (\ref{eq4}) as
\be{}
(-1)^{[i][n]}A^n_i A^i_m - s_{im} A^{m}_m A^n_{m}-
s_{i\si(m)} A^{\si(m)}_m A^n_{\si(m)}=
(s_{ni} - s_{im}) A^i_i A^n_m,
\label{eq4'}
\ee{}
where $m\not=i\not=n \not=m$.

Supposing $n\not =\si(m)$ we find
$(-1)^{[i][n]}A^n_i A^i_m - s_{i\si(m)} A^{\si(m)}_m A^n_{\si(m)}= 0.$
If $i=\si(m)$, then, having in mind
$s_{ii}=(-1)^{[i]}q^{(-1)^{[i]}}\not=\pm 1$ and $A^{\si(m)}_m=y_m\not =0$, we obtain
$A^n_{\si(m)}=0$. Since  $n\not =\si(m)=i$ and $n\not =m$,
this means
either ${\si(m)}\not \in Y$ or ${\si^2(m)}=m$. Assuming
$i\not=\si(m)$, we come to the equation
$s_{i\si(m)} A^{\si(m)}_m A^n_{\si(m)}=0$, which is fulfilled
as well, due to $A^n_{\si(m)}=0$.

It remains to study the case  $n =\si(m)$.
Under this hypothesis, the term $A^n_i A^i_m$ vanishes.
Indeed, since $i\not= m$,
one has $A^i_m \not =0 \Rightarrow i= \si(m)$. But this
contradicts the condition $i\not= n =\si(m)$.
In terms of the variables $x_i$ and $y_i$, equation (\ref{eq4'}) then
reads
$$
- s_{im} x_m y_m-s_{i\si(m)} y_m x_{\si(m)}=
(s_{\si(m)i} - s_{im}) x_i y_m,
$$
where $m\not=i\not =\si(m)$.

Depending on allocation of the indices $i$, $m$, and $\si(m)$,
this equation splits into the following four implications.
\be{}
i<m\quad\mbox{and} \quad i<\si(m) &\Longrightarrow&  x_i=x_m+x_{\si(m)}
,\label{eq41}\\
i<m\quad\mbox{and} \quad i>\si(m) &\Longrightarrow&  x_m=0
\label{eq42},\\
i>m\quad\mbox{and} \quad i<\si(m) &\Longrightarrow&  x_{\si(m)}=0
\label{eq43},\\
i>m\quad\mbox{and} \quad i>\si(m) &\Longrightarrow&  x_i=0
\label{eq44}.
\ee{}
Recall that the index $m$ is assumed to be from $Y$.
Equation (\ref{eq44}) is equivalent to $x_i = 0$ for $i>b_+$.
By Lemma \ref{l-eq31}, this is also true for $i\geq b_+$.
Equations (\ref{eq42}) and  (\ref{eq43}) are thus satisfied
as well.
Equation (\ref{eq41}) states
$x_i = x_{b_-}+x_{\si(b_-)}$
if $i<b_- \in Y_-$
or
$x_i = x_{b_-}+x_{\si^{-1}(b_-)}$
if $i<b_-\in \si(Y_+)$.
Applying Lemma \ref{l-eq31}, we find
$x_i = x_{b_-}$, in either cases.
\end{proof}
\begin{remark}
One can see that the proof of this lemma is almost literally the same as in the non-graded case.
\end{remark}
\begin{lemma}
\label{l-eq3}
Equation (\ref{eq3'}) is fulfilled
by virtue of Lemma \ref{l-eq4}.
\end{lemma}
\begin{proof}
Consider the case $m<i$. Equation  (\ref{eq3'}) holds
if $i\not \in Y$, because the sum turns into $A^i_iA^m_i$.
So we may assume $i \in Y$ and distinguish two
cases: $i\in Y_-$ and $i\in Y_+$.
Assumption $i<\si(i)$ leads to
$A^i_i A^m_i + A^{\si(i)}_i A^m_{\si(i)}=0$.
The equality $m={\si(i)}$ is impossible, since otherwise
$m={\si(i)}<i<{\si(i)}$. Therefore the first term
vanishes and we come to $A^m_{\si(i)}=0$, $m\not= i$.
This condition is satisfied, by Lemma  \ref{l-eq4},
Statement 1.
The case  $i\in Y_+$ results in
$x_i A^m_i=0$. Setting $m =\si(i)$, we come to $ x_i =0 $.
Once $i\in Y_+\Rightarrow i\geq b_+$, we
encounter a particular case of Lemma \ref{l-eq4},
Statement 2.

We should study the situation $i<m$. Equation
(\ref{eq3'})  holds if $\si(i)<m$ because
two possibly non-zero terms $A_i^\nu$ with $\nu=i$ and $\nu=\si(i)$ are not in the range of summation $\nu\geqslant m$.

We may think that  $m\leq\si(i)$; then
$$
\begin{array}{llrll}
i<m=\si(i)& \Rightarrow & A^m_m = 0& \Rightarrow & x_{\si(i)}=0,\\
i<m<\si(i)& \Rightarrow &  A^m_{\si(i)} = 0& \Rightarrow &
\si(i)\not \in Y.
\end{array}
$$
These requirements are fulfilled, by Lemma  \ref{l-eq4}.
\end{proof}

It remains to satisfy equation (\ref{eq5}), in order to complete the
proof of Theorem \ref{clth}.
Let $Y_0$ be the subset in $Y$, such that $\si$ restricted to $Y_0$ is
involutive. By Lemma \ref{l-eq4}, either $\si(i) \not \in Y$
or $i$ and $\si(i)$ belong to $Y_0$ simultaneously.
Equation (\ref{eq5}) brakes down  into four equations
\be{}
       x_m(x_i-x_m) & = & 0 , \hspace{134pt}  i\not \in Y_0,\> m\not \in Y_0,
\label{eq5'}\\
\omega x_m(x_i-x_m) & = & s_{i\si(m)} y_m y_{\si(m)},\quad
                  \hspace{66pt} i \not\in Y_0,\> m \in Y_0,
\label{eq5''}\\
\omega x_m(x_i-x_m) & =  &-s_{m\si(i)} y_i y_{\si(i)} ,\quad
                  \hspace{67pt} i \in Y_0,\> m \not \in Y_0,
\label{eq5'''}\\
\omega x_m(x_i-x_m) & =  & s_{i\si(m)} y_m y_{\si(m)}
- s_{m\si(i)} y_i y_{\si(i)} ,\quad
        i \in Y_0,\> m \in Y_0.
\label{eq5''''}
\ee{}
Everywhere $i<m$, see (\ref{eq5}).
\begin{lemma}
\label{l-eq51}
Suppose $Y_0\not=\varnothing$.
Then $Y=Y_0$ and, moreover,
\be{}
Y_-=\{1,\ldots,b_-\} ,\quad Y_+=\{b_+,\ldots,b_+ +b_- -1\},
\label{notempty}
\ee{}
\be{}
\si(i) = b_+ + b_- -i.
\ee{}
\end{lemma}
\begin{proof}
Let $k \in Y_0 \cap Y_-$ and  $\si(k)\in Y_0 \cap Y_+$ .
Suppose either $l \in Y_-\backslash Y_0$
or $l<\min(Y_-)$.
The assumption
$l<k$ contradicts  equation (\ref{eq5''})
if one sets $i=l$, $m=k$. The
inequality $k<l$ does not agree with (\ref{eq5'''})
if one sets $i=k$, $m=l$.
In both cases one uses Lemma \ref{l-eq4}, Statement 2, and gets
$0=y_k y_{\si(k)}$; that is impossible since
$y_k\not=0$ for all $k\in Y$, by definition of $Y$.
Therefore $Y_-\subset Y_0$ and
$\min(Y_-)=1$.
Assume now  $l \in Y_+\backslash Y_0$. Then,
$l$ cannot exceed $\si(k)$,
because otherwise
$\si(l) <\si^2(k) = k \Rightarrow\si(l)\in Y_0\Rightarrow l\in Y_0$,
an absurd.
The only possibility is $l<\si(k)$. But this again contradicts
equation  (\ref{eq5'''}) if one sets $i=k$ and $m=l$.

Let us prove that $Y_+$ is the integer interval $[b_+,b_++b_--1]$.
Suppose the opposite, then there are $i\in Y_-$ and $m\not\in  Y_+$ such that $b_+<m<\si(i)\in Y_+$.
The equation (2.6) turns to
$
0=-s_{m\si(i)}y_iy_{\si(i)}=-\omega y_iy_{\si(i)},
$
which is impossible.
\end{proof}
\begin{lemma}
\label{l-eq52}
Suppose $Y_0\not=\varnothing$.
Then, there is $a\in \C$ such that $y_i y_{\si(i)} = a$ for all $i\in Y$.
If $(b_-,b_+)\not=\varnothing$, then   $x_m=\la$ for all $m \in (b_-,b_+)$,
where $\la$ is a solution
to the quadratic equation $\la(\la-x_{b_-})=a$.
\end{lemma}
\begin{proof}
By Lemmas \ref{l-eq4} and \ref{l-eq51},
equation (\ref{eq5''''}) is fulfilled if and only if
the product $y_i y_{\si(i)}$ does not depend on $i\in Y$
and is equal to some $a \in \C$.
If  $(b_-,b_+)\not=\varnothing$, equation  (\ref{eq5'''})
suggests $x_m(x_m-x_{b_-})=a$ as soon as $m\in (b_-,b_+)$.
It is easy to see that equation  (\ref{eq5'}) is equivalent
to $x_m=x_i=\la$
for $i,m\in (b_-,b_+)$ and some $\la\in \C$.
\end{proof}
It is convenient to introduce the parameterization  $x_{b_-}=\la+\mu$,
$a=-\mu\la$.
Then $x_m = \la + \mu$ for $m\leq b_-$,
$x_m= \la$ for  $b_-<m<b_+$, and $y_i y_{\si(i)}=-\la\mu$.
This parameterization makes sense even if  $(b_-,b_+)=\varnothing$.
The scalars $\la$ and $\mu$ will have  the meaning of
eigenvalues of the matrix $A$.
\begin{lemma}
Let $A$ be a matrix satisfying equations (\ref{eq1})-(\ref{eq4})
and $(Y,\si)$ the corresponding admissible pair.
Then, equation (\ref{eq5}) gives rise to the
alternative
\begin{itemize}
\item
$Y_0=Y$. $A$ is a solution of Type 1 from  Theorem \ref{clth}.
\item
$Y_0=\varnothing$.
$A$ is a solution of Type 2 from Theorem \ref{clth}.
\end{itemize}
\end{lemma}
\begin{proof}
If $Y_0=\varnothing$, then $Y\cap \si(Y)=\varnothing$ by
Lemma \ref{l-eq4}, Statement 1. Equation
(\ref{eq5}) is reduced to (\ref{eq5'}). It is satisfied if and only if
$x_m=\la$, $i\leq b$, and $x_m=0$, $b< i$, for
some $\la\in \C$ and  $b\in \Z$. Comparing this
with Lemma \ref{l-eq4}, Statement 2,
we come to a solution of Type 2 with  $b\in [b_-,b_+)$.

Suppose $Y_0\not= \varnothing$. Then,
equation (\ref{eq5}) implies Lemmas \ref{l-eq51}
and  \ref{l-eq52}. Conversely, these lemmas
ensure (\ref{eq5'})--(\ref{eq5''''}), which are
equivalent to  (\ref{eq5}). Altogether, this
gives a solution of Type 1.
\end{proof}
Thus we complete the proof of Theorem \ref{clth}

\section{Baxterization}
\label{Sec_Baxt}
In the previous sections, we dealt with so-called constant RE. Its version with spectral parameter
important for applications was  addressed in a number of papers,  e.g. \cite{DK,L}.
Note that consideration was restricted to the following two gradings (standard and symmetric) of the underlying vector space $\C^{m+n}$:
$$
[i]=
\left \{
\begin{array}{ccc}
  0, & i\leqslant m\\
  1, & m<i\leqslant m+n
\end{array}
\right.
,
\quad
[i]=
\left \{
\begin{array}{cccc}
  0, & i\leqslant k &\mbox{or} & k+n<i\\
  1, & k<i\leqslant k+n
\end{array}
\right.,
\quad
\mbox{with}
\quad
m=2k.
$$
The full list of solutions in the standard grading is given \cite{L}.

There is still a problem of clarifying the structure of  solutions with spectral parameter in an arbitrary grading.
In the non-graded case,  every  solution with spectral parameter is obtained from constant  via a
baxterization procedure \cite{KM,IO} up to equivalence, see \cite{RV1}. This equivalence is given by
a family of transformations described in \cite{RV1}, Lemma 3.3.  There is a question  if a similar fact holds true
in an arbitrary grading.
Then the list of constant solutions given in this paper lays a base for such a  classification.
The second step (baxterization) is discussed next.

Recall that a spectral dependent version of the YBE is
$$
S_{23}(x)S_{12}(xy)S_{23}(y) = S_{12}(y)S_{23}(xy)S_{12}(x), \quad x,y\in \C^\times,
$$
where $S$ is regarded as a function $\C\to \End(V)\tp \End(V)$ of (non-zero) spectral parameter.
It is clearly equivalent to a non-graded YBE
\be{}
\label{YBEspec}
\breve{S}_{23}(x)\breve{S}_{12}(xy)\breve{S}_{23}(y) = \breve{S}_{12}(y)\breve{S}_{23}(xy)\breve{S}_{12}(x).
\ee{}
In the $GL$-case, the constant braid matrix $\breve{S}$ satisfies the Hecke condition
$$
\breve{S}^2=\omega \breve{S} +1.
$$
Then $\breve{S}(x)=\breve{S}-x^{-1}\breve{S}^{-1}$ solves  (\ref{YBEspec}), cf. e.g. \cite{Is} and  \cite{M3}.
A baxterization procedure of \cite{KM} makes a constant  RE-matrix $A$ a solution to RE with spectral parameter:
$$
\breve{S}\left(x/y\right)A_2(x)\breve{S}(xy)A_2(y) = A_2(y)\breve{S}(xy)A_2(x)\breve{S}\left (x/y\right).
$$
The minimal polynomial of the  matrix $A$ from Theorem \ref{clth}
is either quadratic or cubic:
$$
(A-\la)(A-\mu), \quad A(A-\la)(A-\mu).
$$
A matrix with cubic minimal polynomial is of Type 1.
Then the  $A(x)$ is, respectively,  either
$$
A(x)=A-\frac{x^{-1}(\xi x-\la-\mu)}{x-x^{-1}}
$$
or
$$
A(x)=A^2+(\xi x-\la-\mu)A-\frac{x^{-1}(x^2\xi^2-\xi x(\la+\mu)+\la\mu)}{x-x^{-1}},
$$
where $\xi$ is an arbitrary parameter.

A universal baxterization formula independent of the minimal polynomial of $A$ can be found in \cite{IO}. It was used for Hecke algebraic formulation of open spin chains
in \cite{Is1}.
In that setting, the matrix $A$ is replaced with the so-called universal matrix $\Kc$ of the quantum general linear supergroup, $U_q(\g)$.
The Hecke algebra admits a representation in $V^{\tp n}\tp \Ac$, where $V=\C^N$ is the natural $U_q(\g)$-module and $\Ac$ a certain $U_q(\g)$-algebra.
This representation takes $\Kc$ to a matrix $K$ whose entries generate $\Ac$.
The assignment $K\mapsto A$ defines a one-dimensional representation of $\Ac$.

The algebra $\Ac$ can be realized as an invariant subalgebra in $U_q(\g)$ (under the adjoint action) generated by the entries of $K=L_-^{-1}L_+$, where
the quantum Lax operators $L_\pm \subset \End(V)\tp U_q(\g)$ are expressed through the universal R-matrix $\Ru$ as $L_+=(\pi\tp \id)(\Ru)$ as $L_-=(\pi\tp \id)(\Ru_{21}^{-1})$.
This factorization of $K$ implies that non-trivial one-dimensional representations of $\Ac$ cannot be restricted
from $U_q(\g)$. A one-dimensional  representation of $U_q(\g)$ assigns $\pm c\in \C$  with an arbitrary sign to natural generators  of the Cartan subalgebra (diagonal entries of $L^\pm$) and kills the root vector generators.
It takes $L_-^{-1}L_+$ to a scalar matrix $A$,
as the only non-vanishing diagonal entries in $L_\pm^{\pm 1}$  produce $c^2$ when squared.
In other words, non-scalar solutions to  RE cannot be obtained from one-dimensional representations of $U_q(\g)$.
This of course applies to the standard quantum supergroup.

In the non-graded case, semi-simple invertible constant solutions to RE classify quantum symmetric pairs via 
a so-called "sandwich" construction introduced in \cite{NDS} and systematically studied in \cite{KS}.
Entries of the matrix $L_-^{-1}AL_+$  generate a left coideal subalgebra $\Bc\subset U_q(\g)$ that  centralizes $A$. 
The dual sandwich $T^{-1}AT$, where $T$ is the RTT-matrix of coordinate functions on the quantum group, satisfies
RE. Its entries generate the subalgebra of $\Bc$-invariants in the Hopf dual to $U_q(\g)$.

\vspace{20pt}.

\noindent
\underline{\large \bf Acknowledgement.}

\vspace{10pt}
\noindent
This work is done at the Center of Pure Mathematics, MIPT, with
financial support of the project FSMG-2023-0013.

The first author (D.A.) is  thankful to the Deanship of Scientific Research at University of Bisha for the financial support through the Scholarship Program of the University.

\section*{Declarations}
\subsubsection*{Data Availability}
 Data sharing not applicable to this article as no datasets were generated or analysed during the current study.

\subsubsection*{Competing interests}
The authors have no competing interests to declare that are relevant to the content of this article.

\end{document}